\documentclass[10pt]{amsart}
\usepackage[utf8]{inputenc}

\title{Random Artin groups}
	\author[Antoine Goldsborough]{Antoine Goldsborough}
	\address{Department of Mathematics, Heriot-Watt University, Edinburgh, UK}
	\email{ag2017@hw.ac.uk}

\author[Nicolas Vaskou]{Nicolas Vaskou}
	\address{Department of Mathematics, Heriot-Watt University, Edinburgh, UK}
	\email{ncv1@hw.ac.uk}

\usepackage{lmodern}

\usepackage[english]{babel}

\usepackage{amsthm}
\usepackage{amsmath}

\usepackage{graphicx}
\usepackage{commath}
\usepackage{dsfont}

\usepackage{tikz}

\usepackage{amssymb}
\usepackage{mathtools} 
\usepackage{float}
\usepackage{setspace}

\theoremstyle{plain}
\newtheorem{lemma}{Lemma}[section]
\newtheorem{proposition}[lemma]{Proposition}
\newtheorem{corollary}[lemma]{Corollary}
\newtheorem{theorem}[lemma]{Theorem}

\newtheorem{problem}[lemma]{Question}

\newtheorem{remark}[lemma]{Remark}
\theoremstyle{definition}
\newtheorem{definition}[lemma]{Definition}
\newtheorem{conjecture}[lemma]{Conjecture}

\begin{document}
\maketitle

\begin{abstract}

We introduce a new model of random Artin groups. The two variables we consider are the rank of the Artin groups and the set of permitted coefficients of their defining graphs.

The heart of our model is to control the speed at which we make that set of permitted coefficients grow relatively to the growth of the rank of the groups, as it turns out different speeds yield very different results. We describe these speeds by means of (often polynomial) functions. In this model, we show that for a large range of such functions, a random Artin group satisfies most conjectures about Artin groups asymptotically almost surely.

Our work also serves as a study of how restrictive the commonly studied families of Artin groups are, as we compute explicitly the probability that a random Artin group belongs to various families of Artin groups, such as the classes of $2$-dimensional Artin groups, $FC$-type Artin groups, large-type Artin groups, and others.

\end{abstract}

\section{Introduction.}

Artin groups are a family of groups that have drawn an increasing interest in the past few decades. They are defined as follows. Let $\Gamma$ be a \textbf{defining graph}, that is a simplicial graph with vertex set $V(\Gamma)$ and edge set $E(\Gamma)$, such that every edge $e_{ab}$ of $\Gamma$ connecting two vertices $a$ and $b$ is given a coefficient $m_{ab} \in \{2, 3, \cdots \}$. Then $\Gamma$ defines an \textbf{Artin group}:
$$A_{\Gamma} \coloneqq \langle \ V(\Gamma) \ | \ \underbrace{aba\cdots}_{m_{ab} \ terms} = \underbrace{bab\cdots}_{m_{ab} \ terms}, \forall e_{ab} \in E(\Gamma) \ \rangle.$$


The cardinality of $V(\Gamma)$, that is the number of \textbf{standard generators} of $A_{\Gamma}$, is called the \textbf{rank} of $A_{\Gamma}$. When $a$ and $b$ are not connected by an edge we set $m_{ab} \coloneqq \infty$.
\medskip

One of the main reasons why Artin groups have become of such great interest is because of the amount of (often easily stated) conjectures and problems about them that are still to be solved. While some of these conjectures are algebraic (torsion, centres), some others are more geometric (acylindrical hyperbolicity, CAT(0)-ness), algorithmic (word and conjugacy problems, biautomaticity), or even topological. Although close to none of these conjectures or problems has been answered in the most general case, there has been progress on each of them. A common theme towards proving these conjectures has been to prove them for smaller families of Artin groups.
\medskip

The goal of this paper is to consider Artin groups with a probabilistic approach. One might wonder “What does a typical Artin group look like?”, and hence want to define a notion of randomness for Artin groups. By computing the different “sizes” of the most commonly studied classes of Artin groups, we give a way to quantify how restrictive these different classes really are. In light of that, our model provides a novel and explicit way of quantifying the state of the common knowledge about the aforementioned conjectures and problems about Artin groups.
\medskip

Although Artin groups are defined using defining graphs, it is not known in general when two defining graphs give rise to isomorphic Artin groups. This problem, known as the \textbf{isomorphism problem}, is actually quite hard to solve even for restrictive classes of Artin groups. With our current knowledge, any (reachable) theory of randomness for Artin groups must then be based on the randomness of defining graphs, and not of the Artin groups themselves.
\medskip




Random right-angled Coxeter (and Artin) groups have been studied by several authors in the literature (\cite{CharneyFarber}, \cite{BHSC}), using the Erdős–Rényi model. While in \cite{CharneyFarber} the authors fix the probability of apparition of an edge as some constant $0 \leq p \leq 1$, in \cite{BHSC} this model is refined: $p = p(N)$ depends on the rank $N$ of the group. That said, these models restrict to right-angled groups, where the associated defining graphs are not labelled. In \cite{Deibel}, the author introduces a model of randomness for Coxeter groups in general. There are similarities between this model and ours, although the former revolves more about making the probabilities of apparition of specific coefficients vary. In particular, this model is not very well suited to provide insights on the “sizes” of the most commonly studied classes of Coxeter and Artin groups. On the contrary, this is a central goal of our model.
\medskip

The two variables that come to mind when thinking about Artin groups are their rank, that is the number of vertices of the defining graph, as well as the choice of the associated coefficients. A first step in the theory is to consider what happens if we restrict ourselves to the family $\mathcal{G}^{N,M}$ of all the defining graphs with $N$ vertices and with coefficients in $\{ \infty, 2, 3, \cdots, M \}$, for some $N \geq 1$ and $M \geq 2$. As we want any possible rank and any possible coefficient to eventually appear in a random Artin group, a convenient way to think about randomness is to pick a defining graph at random in the family $\mathcal{G}^{N,M}$, and then to make $N$ and $M$ grow to infinity.
\medskip

As it turns out, randomness of defining graphs highly depends on the speed at which $N$ and $M$ grow. A prime example of this is that the probability for a defining graph of $\mathcal{G}^{N,M}$ to give an Artin group of large-type (meaning that none of the coefficients is $2$) tends to $1$ when $M$ grows much faster than $N$, and tends to $0$ when $N$ grows much faster than $M$. To solve this problem, we decide to relate $N$ and $M$ through a function $f$ so that $M \coloneqq f(N)$. This way, we only have to look at the family $\mathcal{G}^{N, f(N)}$ when $N$ goes to infinity.
\medskip

If $A_\mathcal{F}$ is a family of Artin groups coming from a family of defining graphs $\mathcal F$, a way of measuring the “size” of $A_\mathcal{F}$ is to compute the limit
$$\lim_{N \rightarrow \infty} \frac{\#(\mathcal{F} \cap \mathcal{G}^{N, f(N)})}{\#(\mathcal{G}^{N, f(N)})}.$$
Of course, this ratio depends on the choice we make for the function $f$. When the above limit is $1$, that is when the probability that a graph picked at random in $\mathcal{G}^{N, f(N)}$ will give an Artin group that belongs to the said family $A_\mathcal F$ tends to $1$, we say that an Artin group picked at random (relatively to $f$) is \textbf{asymptotically almost surely} in $ A_\mathcal F$.
\medskip

That said, there are families $A_\mathcal F$ of Artin groups for which the above limit tends to $1$ no matter what (sensible) choice we make for the function $f$. We say that such a family is \textbf{uniformly large} (resp. \textbf{uniformly small} if that limit is always $0$). Our first result concern such families of Artin groups:

\begin{theorem}
    The family of irreducible Artin groups and the family of Artin groups with connected defining graphs are uniformly large. On the other hand, the family of Artin groups of type FC is uniformly small. In particular, the same applies to the families of RAAGs and triangle-free Artin groups.

\end{theorem}


As mentioned earlier, there are numerous families of Artin groups whose “size” depends on the choice of function $f$. When $f$ is large enough, which means that the choice of possible coefficients for the defining graphs grows fast enough compared to the rank of the Artin group, we obtain much stronger results. This is made explicit in the next two theorems.
\medskip

Before stating these results, we recall a very natural partial ordering on (non-decreasing divergent) functions, which is given by $f \succ g$ whenever $\lim_{N \rightarrow \infty} f(N)/g(N) = \infty$.

\begin{theorem}
    Let $A_\mathcal{F}$ be any family of Artin groups defined by forbidding a finite number $k$ of coefficients from their defining graphs, and consider a function $f : \mathbf{N} \rightarrow \mathbf{N}$. Let $\Gamma$ be a graph picked at random in $\mathcal{G}^{N, f(N)}$. Then:
    \begin{enumerate}

    \item  If $f(N) \succ N^2$, then $A_\Gamma$ asymptotically almost surely belongs to $\mathcal{F}$.
    \item If $f(N) \prec N^2$, then $A_\Gamma$ asymptotically almost surely does not belong to $\mathcal{F}$.
    \item If $f(N) = N^2$, then the probability that $A_{\Gamma}$ belongs to $\mathcal{F}$ tends to $e^{-k/2}$ when $N \rightarrow \infty$.
   \end{enumerate}
\end{theorem}

Note that the previous theorem applies to the families of large-type, extra-large-type, or large-type and free-of-infinity Artin groups. There are strong results in the literature about these families of Artin groups, as most of the famous conjectures and problems about Artin groups have been solved for at least one of them (see Section 2).
\medskip

While these different families of Artin groups have the same threshold at $f(N) = N^2$ no matter how many coefficients we forbid, the class of $2$-dimensional Artin groups turns out to be substantially bigger. Studying this class, we obtain the following result:

\begin{theorem}
    Consider a non-decreasing divergent function $f : \mathbf{N} \rightarrow \mathbf{N}$. Let $\Gamma$ be a graph picked at random in $\mathcal{G}^{N, f(N)}$. Then:
    \begin{enumerate}

        \item If $f(N) \succ N^{3/2}$, then $A_\Gamma$ asymptotically almost surely is 2-dimensional.
        \item  If $f(N) \prec N^{3/2}$, then $A_\Gamma$ asymptotically almost surely is not 2-dimensional.
    \end{enumerate}
\end{theorem}

A consequence of the two previous theorems is that we are able, when $f$ grows fast enough, to show that an Artin group picked at random asymptotically almost surely satisfies most of the main conjectures about Artin groups:

\begin{theorem} \label{Thm4}
Let $f : \mathbf{N} \rightarrow \mathbf{N}$ be such that $f(N) \succ N^{3/2}$, and let $\Gamma$ be a graph picked at random in $\mathcal{G}^{N, f(N)}$. Then asymptotically almost surely, the following properties hold:

\begin{enumerate}
\item  $A_{\Gamma}$ is torsion-free;
\item  $A_{\Gamma}$ has trivial centre;
\item $A_{\Gamma}$ has solvable word and conjugacy problem;
\item $A_{\Gamma}$ satisfies the $K(\pi, 1)$-conjecture;
\item The set of parabolic subgroups of $A_{\Gamma}$ is closed under (arbitrary) intersections;
\item $A_{\Gamma}$ is acylindrically hyperbolic;
\item $A_{\Gamma}$ satisfies the Tits Alternative.
\smallskip

Moreover, if $f(N) \succ N^2$ then asymptotically almost surely the following properties also hold:
\smallskip

\item $A_{\Gamma}$ is CAT(0);
\item $A_{\Gamma}$ is hierarchically hyperbolic;
\item $A_{\Gamma}$ is systolic and thus biautomatic;
\item $Aut(A_{\Gamma}) \cong A_{\Gamma} \rtimes Out(A_{\Gamma})$, where $Out(A_{\Gamma}) \cong Aut(\Gamma) \times (\mathbf{Z} / 2 \mathbf{Z})$ is finite.
\end{enumerate}

\end{theorem}

At last, we also prove interesting results for families of Artin groups in which the number $M$ of permitted coefficients grows “slowly enough” compared to the rank $N$. We focus on the class of Artin groups $A_{\Gamma}$ whose associated graphs $\Gamma$ are not cones, and we prove that for most (non-decreasing divergent) functions, the probability that a random Artin group is acylindrically hyperbolic and has trivial centre tends to $1$.


\begin{theorem} \label{TheoremNThreshold}
Let $\alpha \in (0,1)$ and let $f : \mathbf{N} \rightarrow \mathbf{N}$ be a non-decreasing divergent function satisfying $f(N) \prec N^{1-\alpha}$. Let now $\Gamma$ be a graph picked at random in $\mathcal{G}^{N, f(N)}$. Then the associated Artin group $A_{\Gamma}$ is acylindrically hyperbolic and has trivial centre asymptotically almost surely.
\end{theorem}

Even though the order on (non-decreasing divergent) function is not total, the results of the above theorems for polynomial functions can be encapsulated in Figure \ref{fig:order_and_properties}.
\medskip

\begin{figure}
\label{fig:order_and_properties}
    \centering
    \includegraphics[scale=0.6]{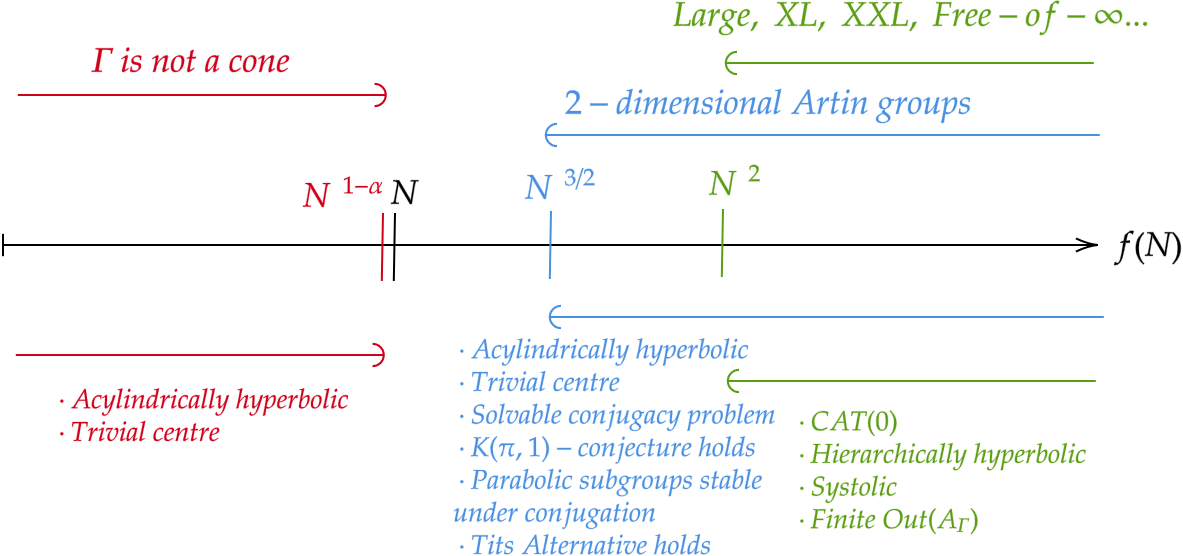}
    \caption{The axis represents various (polynomial) functions $f$. Above the main axis are described the classes of Artin groups that we obtain asymptotically almost surely relatively to $f$, while under this axis we list the properties that we know these groups will satisfy asymptotically almost surely.}
    \label{fig:my_label}
\end{figure}

The previous results shows that we are very close to being able to state that “almost all Artin groups are acylindrically hyperbolic and have trivial centres”. That said, there is a small range of functions for which no probabilistic result regarding acylindrical hyperbolicity or centres being trivial can be stated. In light of that, we raise the following problem:

\begin{problem}
    Construct a family $A_\mathcal{F}$ of acylindrically hyperbolic Artin groups or of Artin groups with trivial centres for which the following holds:
    \\There exists an $\alpha \in (0,1)$ such that for all functions $f : \mathbf{N} \rightarrow \mathbf{N}$ satisfying $N^{1-\alpha} \preccurlyeq f(N) \preccurlyeq N^{3/2}$, a graph $\Gamma$ picked at random in $\mathcal{G}^{N, f(N)}$ is such that $A_{\Gamma}$ asymptotically almost surely belongs to $A_\mathcal{F}$.
\end{problem}

\subsection*{Acknowledgements}

The authors would like to thank Alessandro Sisto for helpful conversations and for pointing out the strategy of the second moment method and Alexandre Martin for useful comments. We would also like to thank the Institut Henri Poincaré (UAR 839 CNRS-Sorbonne Université), and LabEx CARMIN (ANR-10-LABX-59-01) for providing a warm environment for this project to begin. The work of the first author was supported by the EPSRC-UKRI studentship EP/V520044/1.

\section{Preliminaries and first results.}

In this section we bring more details about some of the notions discussed in the introduction. This includes discussions about most of the commonly studied classes of Artin groups, as well as discussions regarding open conjectures related to Artin groups.

Throughout this paper, we will often call a \textbf{triangle} in a graph $\Gamma$ any subgraph of $\Gamma$ that is generated by 3 vertices. This notation will be convenient, although one must note that with this definition, triangles may have strictly fewer than $3$ edges, as subgraphs of $\Gamma$.
\medskip


Most of the main conjectures about Artin groups are still open in general. That said, many of them have been proved for smaller families of Artin groups. Two important of these families are the families of 2-dimensional Artin groups and the family of Artin groups of type FC. These two families have been extensively studied following the work of Charney and Davis (see \cite{CharneyDavis2}). The other well-studied families are usually sub-families of these. 

Before coming to these definitions, we first recall what a parabolic subgroup of an Artin group is. Let $A_{\Gamma}$ be any Artin group, and let $\Gamma'$ be a full subgraph of $\Gamma$. A standard result about Artin groups states that the subgroup of $A_{\Gamma}$ generated by the vertices of $\Gamma'$ is also an Artin group, that is isomorphic to $A_{\Gamma'}$ (\cite{van1983homotopy}). Such a subgroup is called a \textbf{standard parabolic subgroup} of $A_{\Gamma}$. The conjugates of these subgroups are called the \textbf{parabolic subgroups} of $A_{\Gamma}$.


\begin{definition}
\label{defn:2dim}
(0) An Artin group $A_{\Gamma}$ is said to be \textbf{spherical} if the associated Coxeter group $W_{\Gamma}$ is finite.

(1) An Artin group $A_{\Gamma}$ is said to be \textbf{2-dimensional} if for every triplet of distinct standard generators $a, b, c \in V(\Gamma)$, the subgraph $\Gamma'$ spanned by $a$, $b$ and $c$ corresponds to an Artin group $A_{\Gamma'}$ that is \textbf{not} spherical. By a result of (\cite{CharneyDavis2}), this is equivalent to requiring that
$$\frac{1}{m_{ab}} + \frac{1}{m_{ac}} + \frac{1}{m_{bc}} \leq 1.$$
The family of $2$-dimensional Artin groups contains the well-studied families of \textbf{large-type} Artin groups (every coefficient is at least $3$), \textbf{extra-large-type} Artin groups (every coefficient is at least $4$), or \textbf{XXL} Artin groups (every coefficient is at least $5$).

(2) An Artin group $A_{\Gamma}$ is said to be of \textbf{type FC} if every complete subgraph $\Gamma' \subseteq \Gamma$ generates an Artin group $A_{\Gamma'}$ that is spherical. The family of Artin groups of type FC contains the family of right-angled Artin groups, also called \textbf{RAAGs} (the only permitted coefficients are $2$ and $\infty$), the family of spherical Artin groups, and the family of \textbf{triangle-free} Artin groups (the Artin groups whose associated graphs don't contain any $3$-cycles). Being triangle-free is actually equivalent to being both of type $FC$ and $2$-dimensional.
\end{definition}

We now move towards the main conjectures related to Artin groups. For each conjecture, we will briefly describe the state of the common research towards proving it, by mentioning the one or two result(s) that will turn out to be the more “probabilistically relevant” in our model. In other words, the results that cover the largest classes.


\begin{conjecture} \label{Conj1to11}
Let $A_{\Gamma}$ be any Artin group. Then:
\begin{enumerate}

  \item  $A_{\Gamma}$ is torsion-free.

$\hookrightarrow$ This was proved for $2$-dimensional Artin groups (\cite{CharneyKPI}).

  \item  If $A_{\Gamma}$ is irreducible and non-spherical, then $A_{\Gamma}$ has trivial centre.

$\hookrightarrow$ This was proved for $2$-dimensional Artin groups (\cite{Vaskou}), and for Artin groups whose graph is not the cone of a single vertex (\cite{CMW}).

  \item  $A_{\Gamma}$ has solvable word and conjugacy problems.

$\hookrightarrow$ This was proved for $2$-dimensional Artin groups (\cite{systolic})

  \item  $A_{\Gamma}$ satisfies the $K(\pi, 1)$-conjecture.

$\hookrightarrow$ This was proved for $2$-dimensional Artin groups (\cite{CharneyKPI}).

  \item  Intersections of parabolic subgroups of $A_{\Gamma}$ give parabolic subgroups of $A_{\Gamma}$.

$\hookrightarrow$ This was proved for large-type Artin groups (\cite{cumplido_martin_vaskou_2022}) and more generally for $(2, 2)$-free $2$-dimensional Artin groups (\cite{Blufstein}).

  \item $A_{\Gamma}$ is CAT(0).

$\hookrightarrow$ This was proved for XXL Artin groups (\cite{Haettel}).

  \item  If $A_{\Gamma}$ is irreducible and non-spherical, then $A_{\Gamma}$ is acylindrically hyperbolic.

$\hookrightarrow$ This was proved for $2$-dimensional Artin groups ( \cite{Vaskou}), and for Artin groups whose graph is not the cone of a single vertex (\cite{KatoOguni})

  \item  $A_{\Gamma}$ is hierarchically hyperbolic.

$\hookrightarrow$ This was proved for extra-large Artin groups.\cite{HHGS}

  \item  $A_{\Gamma}$ is systolic and biautomatic.

$\hookrightarrow$ This was proved for large-type Artin groups. \cite{huangosajda}

  \item  $A_{\Gamma}$ satisfies the Tits Alternative.

$\hookrightarrow$ This was proved for $2$-dimensional Artin groups (\cite{Martin}).

In addition to these conjectures, the following question has been raised:

  \item  When is $Out(A_{\Gamma})$ finite?

$\hookrightarrow$ This was proved to be the case for large-type free-of-infinity Artin groups (\cite{VaskouFreeofinfty}).
\end{enumerate}
\end{conjecture}



\begin{definition}
Let $\mathcal{F}$ be a family of defining graphs and let $A_{\mathcal F}$ be the corresponding class of Artin groups. Let $f : \mathbf{N} \rightarrow \mathbf{N}$ be a non-decreasing divergent function. We define the probability that an Artin group $A_\Gamma$ picked at random (relatively to $f$) belongs to $A_\mathcal{F}$ as the following limit, when it exists:
$$\mathbb{P}_f \big[A_\Gamma \in A_{\mathcal{F}}\big] \coloneqq \lim_{N \rightarrow \infty} \mathbb{P} \big[\Gamma \in \mathcal{F} \ | \ \Gamma \in \mathcal{G}^{N, f(N)} \big] = \lim_{N \rightarrow \infty} \frac{\#(\mathcal{F} \cap \mathcal{G}^{N, f(N)})}{\#(\mathcal{G}^{N, f(N)})}.$$
Furthermore, we say that an Artin group $A_\Gamma$ picked at random (relatively to $f$) is \textbf{asymptotically almost surely} in $A_\mathcal{F}$ if $\mathbb{P}_f \big[A_\Gamma \in A_\mathcal{F} \big] = 1$. Similarly, we say that $A_\Gamma$ is \textbf{asymptotically almost surely not} in $A_\mathcal{F}$ if $\mathbb{P}_f \big[A_\Gamma \in A_\mathcal{F} \big] = 0$.
\end{definition}

\begin{definition}
Let $A_{\mathcal{F}}$ be a family of Artin groups. Then we say that $A_{\mathcal{F}}$ is \textbf{uniformly large} if for every non-decreasing divergent function $f : \mathbf{N} \rightarrow \mathbf{N}$, an Artin group $A_{\Gamma}$ picked at random (relatively to $f$) is asymptotically almost surely in $A_{\mathcal{F}}$. We say that $\mathcal{F}$ is \textbf{uniformly small} if $A_{\Gamma}$ is asymptotically almost surely not in $A_{\mathcal{F}}$.
\end{definition}

We now move towards our first results. The first thing we will proved is that the family of irreducible Artin groups and the family of Artin groups with connected defining graphs are uniformly large. This is important as many results regarding Artin groups assume that the corresponding groups are irreducible and/or have a connected defining graph. Our work show that these two hypotheses are very much not restrictive.

\begin{definition}
Let $\Gamma_1$ and $\Gamma_2$ be two defining graphs. The graph $\Gamma_1 *_k \Gamma_2$ is the graph obtained by attaching every vertex of $\Gamma_1$ to every vertex of $\Gamma_2$ by an edge with label $k$ (with $k \in \{\infty, 2, 3, \cdots \}$).

Let now $\Gamma$ be any defining graph. Then $\Gamma$ is called a \textbf{k-join} relatively to $\Gamma_1$ and $\Gamma_2$ is there are two subgraphs $\Gamma_1, \Gamma_2 \subseteq \Gamma$ such that $V(\Gamma_1) \sqcup V(\Gamma_2) = V(\Gamma)$ and such that $\Gamma = \Gamma_1 *_k \Gamma_2$.

We will denote by $\mathcal A_{\mathcal{J}_k}$ the class of Artin groups whose defining graphs decompose as $k$-joins.
\end{definition}

\begin{remark} \label{RemarkConnectedIrreducible}
    (1) If $\Gamma \in \mathcal{J}_2$ then $A_{\Gamma}$ decomposes as a direct product $A_{\Gamma_1} \times A_{\Gamma_2}$ in an obvious way. In that case, $\Gamma$ is called \textbf{reducible}. The class $\mathcal{J}_2^C$ of \textbf{irreducible} defining graphs will be denoted $\mathbf{Irr}$.
    \\(2) If $\Gamma \in \mathcal{J}_{\infty}$ then it is disconnected. The class $\mathcal{J}_{\infty}^C$ of connected defining graphs will be denoted $\mathbf{Con}$.
\end{remark}

\begin{lemma}
	\label{lem:being_reducible_uniformly small}
	The family $\mathcal A_{\mathcal{J}_k}$ is uniformly small. In particular, the classes $A_{Irr}$ and $A_{Con}$ of Artin groups are both uniformly large.
\end{lemma}

\begin{proof}
	We will count the number of decompositions of the graph $\Gamma$ as $\Gamma = \Gamma_1 *_k \Gamma_2$. Without loss of generality, we will let $\Gamma_1$ denote the subgraph with the lower rank, so that $| V(\Gamma_1) | \leq \lfloor N/2 \rfloor$. Let $f : \mathbf{N} \rightarrow \mathbf{N}$ be a non-decreasing divergent function and consider the family $\mathcal J_k$. For a given $N \geq 1$, we have:
		
		\begin{align*}
		\begin{split}
			\mathbb P \big[ \Gamma \in \mathcal{J}_k \ | \ \Gamma \in \mathcal{G}^{N, f(N)} \big] & = \mathbb P\big[ \exists \ \Gamma_1, \Gamma_2 \quad \text{with} \quad  | V(\Gamma_1) | \leq N/2  \quad \text{such that} \quad  \Gamma = \Gamma_1 *_k \Gamma_2 \ | \ \Gamma \in \mathcal{G}^{N, f(N)} \big] \\
			\leq & \sum_{j=1}^{\lfloor N/2 \rfloor} \mathbb P\big[ \exists \ \Gamma_1, \Gamma_2 \quad \text{with} \quad  | V(\Gamma_1) | = j  \quad \text{such that} \quad  \Gamma = \Gamma_1 *_k \Gamma_2 \ | \ \Gamma \in \mathcal{G}^{N, f(N)} \big] \\
			= & \sum_{j=1}^{\lfloor N/2 \rfloor} {N \choose j} \left(\frac{1}{f(N)}\right)^{j(N-j)} \\
			\leq & \sum_{j=1}^{\lfloor N/2 \rfloor} \left(\frac{Ne}{jf(N)^{N/2}}\right)^{j} \\
			\leq \ & \frac{Ne}{f(N)^{N/2}} \cdot \left(\frac{1-\left(\frac{Ne}{f(N)^{N/2}}\right)^{N/2+1}}{1-\frac{Ne}{f(N)^{N/2}}} \right) 
		\end{split}
	\end{align*}

	
where we used the bound ${N \choose j} \leq \left(\frac{Ne}{j}\right)^j$. Now $\lim_{N\to \infty} \frac{Ne}{f(N)^{N/2}} =0$ for any non-decreasing divergent function $f$, so we obtain
$$\mathbb{P}_f \big[ A_{\Gamma} \in A_{\mathcal{J}_k} \big] = \lim_{N \rightarrow \infty} \mathbb P \big[ \Gamma \in \mathcal{J}_k \ | \ \Gamma \in \mathcal{G}^{N, f(N)} \big] = 0 \cdot \left( \frac{1-0}{1-0} \right) = 0.$$
This proves the main statement of the lemma. The second statement then directly follows from Remark \ref{RemarkConnectedIrreducible}.
\end{proof}

Our next result concerns the class of Artin groups of type FC.

\begin{lemma}
    Let $A_{\mathcal{FC}}$ be the family of Artin groups of type FC. Then $A_{\mathcal{FC}}$ is uniformly small. In particular, the family of triangle-free Artin groups, the family of spherical Artin groups and the family of RAAGs are also uniformly small.
\end{lemma}

\begin{proof}
    Let $f$ be any non-decreasing divergent function, and let $\Gamma \in \mathcal{G}^{N,f(N)}$. We want to compute the probability that $\Gamma$ belongs to $\mathcal{FC} \cap \mathcal{G}^{N,f(N)}$. Let $a$, $b$ anc $c$ be three vertices of $\Gamma$. The probability that any of the three corresponding coefficients $m_{ab}$, $m_{ac}$ and $m_{bc}$ is not $2$ nor $\infty$ is precisely $\frac{f(N)-2}{f(N)}$, and hence the probability that the three coefficients are not $2$ nor $\infty$ is $\left( \frac{f(N)-2}{f(N)} \right)^3$. Note that when this happens, the subgraph $\Gamma' \subseteq \Gamma$ spanned by $a$, $b$ and $c$ is complete but generates an Artin group $A_{\Gamma'}$ which is non-spherical (the sum of the inverses of the three corresponding coefficients is $\leq 1$). In particular, $\Gamma$ is not of type FC. We obtain
    $$\mathbb{P}_f\big[A_\Gamma \notin A_\mathcal{FC}\big] = \lim_{N \rightarrow \infty} \frac{\#(\mathcal{G}^{N, f(N)} \backslash \mathcal{FC})}{\#(\mathcal{G}^{N, f(N)})}
    \geq \lim_{N \rightarrow \infty}\left( \frac{f(N)-2}{f(N)} \right)^3 = \lim_{N \rightarrow \infty} \left( 1 - \frac{2}{f(N)} \right)^3 = 1.$$
\end{proof}



As mentioned in the introduction, there are interesting classes of Artin groups for which the probability that a graph taken at random will belong to the class highly depends on the choice of function $f$. Some examples are given through the following theorem.

\begin{theorem} \label{Theorem_N^2}
    Let $A_\mathcal{F}$ be any family of Artin groups defined by forbidding from their graphs a finite number $k$ of coefficients, and consider a function $f : \mathbf{N} \rightarrow \mathbf{N}$. Let $A_\Gamma$ be an Artin group picked at random (relatively to $f$). Then:
    \begin{enumerate}
    \item  If $f(N) \succ N^2$, then $A_{\Gamma}$ asymptotically almost surely belongs to $A_\mathcal{F}$.
    \item If $f(N) \prec N^2$, then $A_{\Gamma}$ asymptotically almost surely does not belong to $A_\mathcal{F}$.
    \item If $f(N)=N^2$ then asymptotically we have $\mathbb P_f \big[A_\Gamma \in A_\mathcal{F}\big] = e^{-k/2}$.
    \end{enumerate}
\end{theorem}

\begin{proof}
A graph with $N$ vertices has $\frac{N(N-1)}{2}$ pairs of vertices, each of which is given one of $f(N)$ possible coefficients. Hence, direct computations on the possible number of graphs give 
$$\# \mathcal G^{N, f(N)}  = (f(N))^{\frac{N(N-1)}{2}}.$$
Similarly, we have
$$\# (\mathcal{F} \cap \mathcal G^{N, f(N)})  = (f(N)-k)^{\frac{N(N-1)}{2}}.$$
And thus we obtain
$$\mathbb{P}_f\big[A_\Gamma \in A_\mathcal{F}\big] = \lim_{N \rightarrow \infty} \frac{\#(\mathcal{F} \cap \mathcal G^{N, f(N)})}{\#(\mathcal{G}^{N, f(N)})} = \lim_{N \rightarrow \infty} \left(\frac{f(N)-k}{f(N)}\right)^{\frac{N(N-1)}{2}} = \lim_{N \rightarrow \infty} \left(\frac{f(N)-k}{f(N)}\right)^{N^2 \left(\frac{N-1}{2N} \right)}.$$
If $f(N) \succ N^2$, there is a function $h$ with $\lim_{N \rightarrow \infty} h(N) = \infty$ such that $f(N)=h(N)N^2$ and hence $$\mathbb{P}_f\big[A_\Gamma \in A_\mathcal{F}\big] = \lim_{N \rightarrow \infty} \left(\frac{f(N)-k}{f(N)}\right)^{f(N) \left( \frac{N-1}{2Nh(N)} \right)} = \lim_{N \rightarrow \infty} e^{-k \left( \frac{N-1}{2Nh(N)} \right)} = 1.$$

\noindent If $f(N) \prec N^2$, there exists a function $h$ with $\lim_{N \rightarrow \infty} h(N) = \infty$ such that $f(N)h(N)=N^2$ and

$$\mathbb{P}_f\big[A_\Gamma \in A_\mathcal{F}\big] = \lim_{N \rightarrow \infty} \left(\frac{f(N)-k}{f(N)}\right)^{f(N) \left( \frac{(N-1)h(N)}{2N} \right)} = \lim_{N \rightarrow \infty} e^{-k \left( \frac{(N-1)h(N)}{2N} \right)} = 0.$$
\medskip

\noindent Finally if $f(N)=N^2$ then
$$\mathbb{P}_f\big[A_\Gamma \in A_\mathcal{F}\big] = \lim_{N \rightarrow \infty} \left(\frac{N^2-k}{N^2}\right)^{N^2 \left(\frac{N-1}{2N} \right)} = \lim_{N \rightarrow \infty} e^{-k \left( \frac{N-1}{2N} \right)} = e^{-k/2}.$$
\end{proof}

The previous theorem has many consequences, as it can be applied to the families of large-type, extra-large-type, XXL or free-of-infinity Artin groups, for which much is known. Before stating an explicit result in Corollary \ref{Coro11Conj}, we prove the following small lemma:

\begin{lemma}
\label{lem:interescting_big_class}
Let $A_\mathcal{F}$ and $A_\mathcal{H}$ be two families of Artin groups, let $f : \mathbf N \rightarrow \mathbf N$ be a non-decreasing divergent function, and suppose that $\mathbb{P}_f \big[A_\Gamma \in A_\mathcal{H} \big] = 1$. Then
$$\mathbb{P}_f \big[A_\Gamma \in A_\mathcal{F} \big] = \mathbb{P}_f \big[A_\Gamma \in A_\mathcal{F} \cap A_\mathcal{H} \big].$$
\end{lemma}



\begin{proof}
This is straightforward:
$$\mathbb P_f\big[ A_\Gamma \in A_\mathcal{F} \big]
= \mathbb P_f\big[ A_\Gamma \in A_\mathcal{F} \cap A_\mathcal{H} \big] + \underbrace{\mathbb P_f\big[ A_\Gamma \in A_\mathcal{F} \cup A_\mathcal{H} \big]}_{=1} - \underbrace{\mathbb P_f\big[ A_\Gamma \in A_\mathcal{H} \big]}_{=1}
= \mathbb P_f\big[ A_\Gamma \in A_\mathcal{F} \cap A_\mathcal{H} \big].$$
\end{proof}




\begin{corollary} \label{Coro11Conj}
Let $f : \mathbf{N} \rightarrow \mathbf{N}$ be a function satisfying $f(N) \succ N^2$. Then an Artin group $A_{\Gamma}$ picked at random (relatively to $f$) satisfies any of the following property asymptotically almost surely:
\begin{enumerate}
    \item $A_{\Gamma}$ is torsion-free;
   \item $A_{\Gamma}$ has trivial centre;
   \item $A_{\Gamma}$ has solvable word and conjugacy problems;
    \item $A_{\Gamma}$ satisfies the $K(\pi, 1)$-conjecture;
    \item The set of parabolic subgroups of $A_{\Gamma}$ is closed under arbitrary intersections;
    \item $A_{\Gamma}$ is CAT(0);
   \item $A_{\Gamma}$ is acylindrically hyperbolic;
   \item $A_{\Gamma}$ is hierarchically hyperbolic;
    \item $A_{\Gamma}$ is systolic and biautomatic;
    \item $A_{\Gamma}$ satisfies the Tits Alternative;
    \item $Aut(A_{\Gamma}) \cong A_{\Gamma} \rtimes Out(A_{\Gamma})$, where $Out(A_{\Gamma}) \cong Aut(\Gamma) \times (\mathbf{Z} / 2 \mathbf{Z})$ is finite.
    \end{enumerate}
\end{corollary}


\begin{proof}
Let $A_{\mathcal{K}}$ be the class of XXL free-of-infinity Artin groups, and let $A_{\mathcal{L}} \coloneqq A_{Irr} \cap A_{Con} \cap A_{\mathcal{K}}$. Using Lemma \ref{lem:being_reducible_uniformly small} and Lemma \ref{lem:interescting_big_class} we can see that
$\mathbb P_f\big[A_\Gamma \in A_{\mathcal{L}} \big]=  \mathbb P_f \big[ A_\Gamma \in A_{\mathcal{K}}\big]$. By Theorem \ref{Theorem_N^2}, an Artin group $A_{\Gamma}$ picked at random (relatively to $f$) is asymptotically almost surely in $A_{\mathcal{L}}$. The various results given in Conjecture \ref{Conj1to11} concern families of Artin groups that all contain the family $A_{\mathcal{L}}$. In particular, every Artin group in $A_{\mathcal{L}}$ satisfies the 11 points of the Corollary.
\end{proof}

\section{Two-dimensional Artin groups.}


This section aims at studying from our probabilistic point of view the family of $2$-dimensional Artin groups. This family is particularly important in the study of Artin groups, and many authors in the literature have obtained strong results for this class (see Conjecture \ref{Conj1to11}).

Our goal will be to show that if $f(N) \succ N^{3/2}$ then asymptotically almost surely a random Artin group (relative to $f$) will be $2$-dimensional and if $f(N) \prec  N^{3/2}$ then asymptotically almost surely a random Artin group (relative to $f$) will not be $2$-dimensional. In particular, we will be able to improve the result of Corollary \ref{Coro11Conj}, thus proving Theorem \ref{Thm4}.

The condition of being $2$-dimensional (see Definition \ref{defn:2dim}.(1)) is quite specific, which makes it hard to compute the “size” of the family. As it turns out, the size of this family is comparable to the size of another family of Artin groups, which will turn out to be easier to compute (see Lemma \ref{lem:comparing_dim2_2_2free} and Theorem \ref{theorem:size_of_dimension2}). This other family resembles the family introduced in \cite{Blufstein}. We introduce it thereafter:

\begin{definition}
\label{defn:2free}
We say an Artin group $A_\Gamma$ is $(2,2)$--free if $\Gamma$ does not have any two adjacent edges labelled by $2$. We denote by $A_{\mathcal B}$ the family of $(2,2)$-free Artin groups.
\end{definition}


The following lemma is a key result. It will allow us to restrict to the study of $(2,2)$-free Artin groups, as asymptotically this family has the same size as the family $A_\mathcal{D}$ of $2$-dimensional Artin groups.

\begin{lemma}
\label{lem:comparing_dim2_2_2free}
	For all non-decreasing divergent functions $f:\mathbf N \to \mathbf N$,  we have:
    \begin{itemize}
        \item $\mathbb P_f \big[ A_\Gamma \in A_\mathcal D\big] \leq \mathbb P_f \big[ A_\Gamma \in A_\mathcal B\big]$;
        \item Further, if $f(N) \succ N$, then $\mathbb P_f \big[ A_\Gamma \in \mathcal A_{\mathcal D} \big]=\mathbb P_f \big[ A_\Gamma \in A_\mathcal B\big]$.
    \end{itemize}
\end{lemma}







\begin{proof}

The probability that a defining graph $\Gamma$ picked at random gives rise to a $2$-dimensional Artin group can be found by conditioning on the event “$\Gamma \in \mathcal B$”:

\begin{align*}
   (*) \ \ \ \ \mathbb P \big[ \Gamma \in \mathcal D \mid \Gamma \in \mathcal G^{N, f(N)}\big] \ = \ & \mathbb P \big[ \Gamma \in \mathcal D \ | \ (\Gamma \in \mathcal B) \cap (\Gamma \in \mathcal G^{N, f(N)})\big]\mathbb P \big[\Gamma \in \mathcal B\mid \Gamma \in \mathcal G^{N, f(N)}\big] \\
   + \ &\mathbb P \big[ \Gamma \in \mathcal D \ | \ (\Gamma \not\in \mathcal B) \cap (\Gamma \in \mathcal G^{N, f(N)})\big]\mathbb P \big[ \Gamma \not\in \mathcal B \mid \Gamma \in \mathcal G^{N, f(N)}\big] 
\end{align*}

Note that once we have two adjacent edges $e_1, e_2$ labelled by 2, then the probability that the triangle spanned by $\{e_1, e_2\}$ generates an Artin groups of spherical type is exactly the probability that the last edge is not labelled by $\infty$. This probability is  $\frac{f(N)-1}{f(N)}$, hence we have
$$\mathbb P \big[ \Gamma \in \mathcal D \ | \ (\Gamma \not\in \mathcal B)\cap (\Gamma \in \mathcal G^{N, f(N)})\big] \leq 1 - \frac{f(N)-1}{f(N)} = \frac{1}{f(N)}.$$

Whence we get the following upper bound for $(*)$:
$$\mathbb P\big[ \Gamma  \in \mathcal D \mid \Gamma \in \mathcal G^{N, f(N)}\big]  \leq \mathbb P \big[ \Gamma \in \mathcal B \mid \Gamma \in \mathcal G^{N, f(N)}\big]+\mathbb P \big[ \Gamma \not\in \mathcal B \mid \Gamma \in \mathcal G^{N, f(N)}\big] \cdot \frac{1}{f(N)}.$$

We now deal with the lower bound. The probability that a given triangle $\Delta$ is not of spherical type is the quotient
$$\frac{\# \text{ ways that } \Delta \text{ can be spherical}}
{\# \text{ possible coefficients on } \Delta}. \ \ \ \ (**)$$
In our case, it is given that $A_{\Gamma}$ is $(2, 2)$-free, so the only triangles which are not of spherical type are of the form $(2,3,3)$; $(2,3,4)$ or $(2,3,5)$. When considering the possible permutations of the order of the coefficients, this gives $15$ possibilities. This gives the numerator of $(**)$.

A clear upper bound for the denominator of $(**)$ is $f(N)^3$. However, it may not be equal to $f(N)^3$, as the condition of being $(2, 2)$-free coming from other edges in the graph could force some edges of $\Delta$ to not be labelled by $2$. That said, the only triplet of coefficients for $\Delta$ that could be forbidden by adjacent edges would be those containing at least one edge labelled with $2$. This number is bounded by $3N$, thus the denominator of $(**)$ admits $f(N)^3 - 3N$ as lower bound.

Putting everything together, we obtain
$$\frac{15}{f(N)^3} \leq
\frac{\# \text{ ways that } \Delta \text{ can be spherical}}
{\# \text{ possible coefficients on } \Delta}
\leq \frac{15}{f(N)^3 - 3N}. \ \ \ \ (***)$$

Hence, by an union bound we get:
\begin{align*}
    \mathbb P\big[ \Gamma  \not\in \mathcal D \ | \ (\Gamma \in \mathcal B)\cap (\Gamma \in \mathcal G^{N, f(N)}) \big]
    &\leq \sum_{\Delta \; \text{triangle in} \; \Gamma}\mathbb P \big[ \Delta \;\text{is of spherical type} \;\ | \ (\Gamma \in \mathcal B) \cap (\Gamma \in \mathcal G^{N, f(N)})  \big] \\
    &\leq {N \choose 3} \frac{15}{f(N)^3 -3N}.
\end{align*}

Therefore: $$\mathbb P\big[ \Gamma  \in \mathcal D \mid  \Gamma \in \mathcal G^{N, f(N)}\big] \geq \left(1-{N \choose 3}\frac{15}{f(N)^3 -3N}\right)\mathbb P \big[ \Gamma \in \mathcal B \mid \Gamma \in \mathcal G^{N, f(N)}\big]. \ \ \ \ (****)$$



Now for any non-decreasing divergent function $f$ we have that $\frac{1}{f(N)} \to 0$ hence
$$\mathbb P_f \big[ A_\Gamma \in A_\mathcal D\big]=\lim_{N \to \infty} \mathbb P\big[ \Gamma  \in \mathcal D \mid \Gamma \in \mathcal G^{N, f(N)}\big] \overset{(*)} \leq \lim_{N \to \infty} \mathbb P \big[ \Gamma \in \mathcal B \mid \Gamma \in \mathcal G^{N, f(N)}\big]=\mathbb P_f\big[ A_{\Gamma} \in A_{\mathcal B}\big].$$

If $f(N) \succ N$ it is not hard to see that
$$\lim_{N \to \infty} \left({N \choose 3}\frac{15}{f(N)^3 -3N}\right) = 0.$$

This means that
$$\mathbb P_f \big[ A_\Gamma \in A_\mathcal D\big]=\lim_{N \to \infty} \mathbb P\big[ \Gamma  \in \mathcal D \mid \Gamma \in \mathcal G^{N, f(N)}\big] \overset{(****)} \geq \lim_{N \to \infty} \mathbb P \big[ \Gamma \in \mathcal B \mid \Gamma \in \mathcal G^{N, f(N)}\big]=\mathbb P_f\big[ A_{\Gamma} \in A_{\mathcal B}\big].$$
\end{proof}


We now move towards determining for which (non-decreasing divergent) functions an Artin group picked at random is asymptotically almost surely $2$-dimensional, or not $2$-dimensional. In view of Lemma \ref{lem:comparing_dim2_2_2free}, looking at $(2, 2)$-free Artin groups will be enough to give a conclusion for $2$-dimensional Artin groups. The result we want to prove is the following:


\begin{theorem} \label{theorem:size_of_dimension2}
Let $f : \mathbf{N} \rightarrow \mathbf{N}$, and let $A_\Gamma$ be an Artin group picked at random (relatively to $f$). Then:
    \begin{enumerate}
    \item \label{item:bigger} If $f(N) \succ N^{3/2}$, then asymptotically almost surely $A_{\Gamma}$ is $2$-dimensional.
    \item \label{item:smaller} If $f(N) \prec N^{3/2}$, then asymptotically almost surely $A_{\Gamma}$ is not $2$-dimensional.
    \item \label{item:equal} If $f(N)=N^{3/2}$ then then $\mathbb P_f \big[ A_\Gamma \in \mathcal D\big] \leq 2/3.$
    \end{enumerate}
\end{theorem}



\begin{proof}
Let $f$ be any non-decreasing, divergent function. We need to compute $\mathbb P_f \big[ A_{\Gamma} \in A_{\mathcal D} \big]$. In view of Lemma \ref{lem:comparing_dim2_2_2free}, it is enough to compute $\mathbb P_f \big[ A_{\Gamma} \in A_{\mathcal B} \big]$, i.e. the probability that an Artin group $A_{\Gamma}$ picked at random is $(2, 2)$-free. To do this, we will use the second moment method.





Let us consider a graph $\Gamma \in \mathcal{G}^{N, f(N)}$. For any ordered triplet $(v_1, v_2, v_3)$ of distinct vertices of $\Gamma$, we let $ I_{(v_1, v_2, v_3)}:\mathcal G^{N, f(N)} \to \{0,1\}$ be the random variable which takes $1$ on $\Gamma \in \mathcal G^{N, f(N)}$ precisely when  $(v_1, v_2, v_3)$ spans a triangle with $m_{v_1, v_2}=m_{v_1, v_3}=2$. We let
$$X= \left( \sum_{(v_1, v_2, v_3) \in V(\Gamma)^3} I_{(v_1, v_2, v_3)} \right) :\mathcal G^{N, f(N)} \to \mathbf N$$
where the sum is taken over all triplets of distinct vertices. The variable $X$ counts the number of pairs of adjacent edges labelled by a $2$, twice (because of the permutation of these edges).

We can compute the expectation $\mathbb E \big[ I_{(v_1, v_2, v_3)}\big]=f(N)^{-2}$ and hence 

 \begin{align*}
 \begin{split}
 \mathbb E \big[ X\big]&=\sum_{(v_1, v_2, v_3)}\mathbb E \big[ I_{(v_1, v_2, v_3)}\big] \\
 &=N(N-1)(N-2)f(N)^{-2} \\
 &\sim N^3f(N)^{-2}. 
 \end{split}
 \end{align*}
 
 Now, we use the second moment method, as in (\cite{CharneyFarber}, Theorem 6):
 
 $$\mathbb P \big[ X \neq 0\big] \geq \frac{\mathbb E\big[ X \big]^2 }{\mathbb E \big[ X^2\big]}.$$
 
 We have already computed $\mathbb E \big[ X\big]$, so we now compute $ \mathbb E \big[ X^2\big]$ by dividing into several cases the sum
 $$X^2= \sum I_{(v_1, v_2, v_3)}I_{(w_1, w_2, w_3)}.$$
 Note that the sum is taken over all ordered triplets $(v_1, v_2, v_3)$ and $(w_1, w_2, w_3)$ of vertices, where the $v_i$'s are distinct, and the $w_i$'s are distinct. In a triangle $(v_1, v_2, v_3)$ such that $m_{v_1, v_2}=m_{v_1, v_3}=2$, we shall call $v_1$ the $\textbf{central}$ vertex of the triangle.The different cases are treated below. They can be seen in Figure \ref{fig:8cases}.
 \medskip

\textbf{Case 1:} Let $X_1$ denote the sum of products $I_{(v_1, v_2, v_3)}I_{(w_1, w_2, w_3)}$ such that no vertex appears in both triples. Then $$ \mathbb E\big[ X_1\big] =\frac{N!}{(N-6)!}f(N)^{-4} \sim N^{6}f(N)^{-4}.$$
  
\textbf{Case 2:} Let $X_2$ denote the sum of products $I_{(v_1, v_2, v_3)}I_{(w_1, w_2, w_3)}$ such that these two triangles share exactly one vertex and the vertex they share is central in both triangles (i.e. $v_1=w_1$). Then we have
$$ \mathbb E\big[ X_2\big] =\frac{N!}{(N-5)!}f(N)^{-4} \sim N^{5}f(N)^{-4}.$$

\textbf{Case 3:} Let $X_3$ denote the sum of products $I_{(v_1, v_2, v_3)}I_{(w_1, w_2, w_3)}$ such that these two triangles share exactly one vertex, where this vertex is the central vertex for one triangle and not a central vertex for the other triangle (for example $v_2=w_1$). In this case, we get:

$$ \mathbb E\big[ X_3\big] =4\frac{N!}{(N-5)!}f(N)^{-4} \sim 4N^{5}f(N)^{-4}.$$

\textbf{Case 4:} Let $X_4$ denote the sum of products $I_{(v_1, v_2, v_3)}I_{(w_1, w_2, w_3)}$ such that these two triangles share exactly one vertex, where this vertex is not central for either triangle (for example $v_2=w_2$). Then
$$ \mathbb E\big[ X_4\big] =4\frac{N!}{(N-5)!}f(N)^{-4} \sim 4N^{5}f(N)^{-4}.$$
     
 \textbf{Case 5:} Let $X_5$ denote the sum of products $I_{(v_1, v_2, v_3)}I_{(w_1, w_2, w_3)}$ such that these two triangles share exactly two vertices and these two vertices are not central for either triangle (for example $v_2=w_2$ and $v_3=w_3$). In this case 
 $$ \mathbb E\big[ X_5\big] =2\frac{N!}{(N-4)!}f(N)^{-4} \sim 2N^{4}f(N)^{-4}.$$

 \textbf{Case 6:} Let $X_6$ denote the sum of products $I_{(v_1, v_2, v_3)}I_{(w_1, w_2, w_3)}$ such that these two triangles share exactly two vertices and one of these is central in both triangles and the other is not (for example $v_1=w_1$ and $v_3=w_2$). In this case 
 $$ \mathbb E\big[ X_6\big] =4\frac{N!}{(N-4)!}f(N)^{-3} \sim 4N^{4}f(N)^{-3}.$$
 
  \textbf{Case 7:} Let $X_7$ denote the sum of products $I_{(v_1, v_2, v_3)}I_{(w_1, w_2, w_3)}$ such that these two triangles share exactly two vertices where one of these is central for the triangle $(v_1, v_2, v_3)$ but not for $(w_1, w_2, w_3)$, and the other vertex is central for the triangle $(w_1, w_2, w_3)$ but not for $(v_1, v_2, v_3)$ (for example $v_1=w_3$ and $w_1=v_3$). In this case we have:
  
   $$ \mathbb E\big[ X_7\big] =4\frac{N!}{(N-4)!}f(N)^{-3} \sim 4N^{4}f(N)^{-3}.$$
   
\textbf{Case 8:} Let $X_8$ denote the sum of products $I_{(v_1, v_2, v_3)}I_{(w_1, w_2, w_3)}$ such that these two triangles share all three vertices, and such the central vertices of both triangles are the same (i.e. $v_1=w_1$). In this case, we have:

  
   $$ \mathbb E\big[ X_8 \big] =2\frac{N!}{(N-3)!}f(N)^{-2} \sim 2N^{3}f(N)^{-2}.$$
   
\textbf{Case 9:} Let $X_9$ denote the sum of products $I_{(v_1, v_2, v_3)}I_{(w_1, w_2, w_3)}$ such that these two triangles share all three vertices, and such that the central vertex of the first triangle is not the central vertex of the second triangle (for example $v_1 = w_2$). Then the three edges of the triangle must be labelled by a $2$, and we get

$$ \mathbb E\big[ X_9 \big] = 4 \frac{N!}{(N-3)!}f(N)^{-3} \sim 2N^{3}f(N)^{-3}.$$
   
\begin{figure}
\label{fig:8cases}
    \centering
    \includegraphics[scale=0.5]{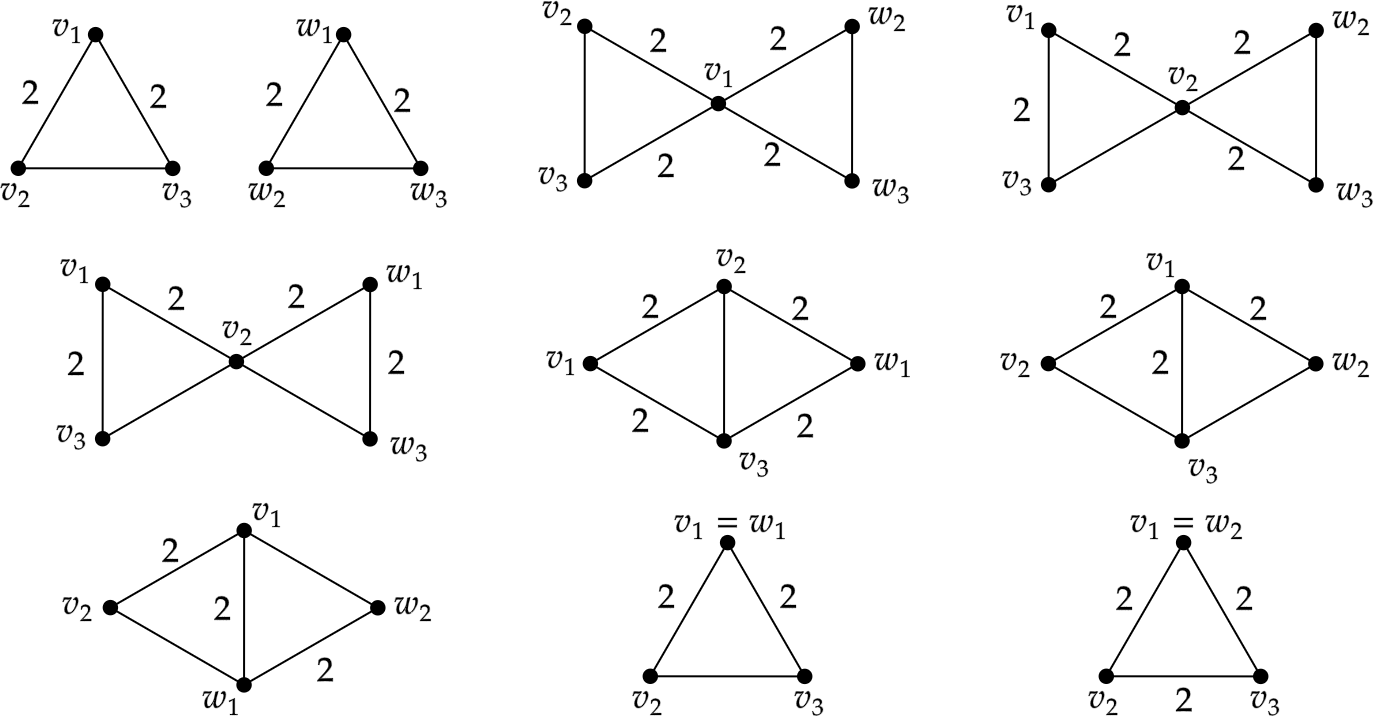}
    \caption{From top-left to bottom-right: the 9 cases described in the proof of Theorem \ref{theorem:size_of_dimension2}. The edges that are not explicitly labelled by $2$ can be labelled by any coefficient, including $\infty$.}
\end{figure}
   
   Therefore, we have

   \begin{align*}
   	\begin{split}
   		\frac{\mathbb E\big[ X^2 \big] }{\mathbb E \big[ X\big]^2} &= \sum_{i=1}^{8} \frac{\mathbb E\big[ X_i \big] }{\mathbb E \big[ X\big]^2} \\
   		&\sim\frac{N^6f(N)^{-4}+9N^5f(N)^{-4}+2N^4f(N)^{-4}+8N^4f(N)^{-3}+2N^{3}f(N)^{-3}+2N^3f(N)^{-2}}{N^6f(N)^{-4}} \\
   		&\sim 1+\frac{9}{N}+\frac{2}{N^2}+\frac{8f(N)}{N^2}+\frac{4f(N)}{N^3}+\frac{2f(N)^2}{N^3}.
   	\end{split}
   \end{align*}
   
  Hence, if $f(N) \prec N^{3/2}$ then by definition there exists a non-decreasing divergent function $h$ such that $f(N)h(N)=N^{3/2}$. In this case we get:
  
  $$ \mathbb P \big[ X \neq 0\big] \geq \left(\frac{\mathbb E\big[ X^2 \big] }{\mathbb E \big[ X\big]^2}\right)^{-1} \sim \left(1+\frac{9}{N}+\frac{2}{N^2}+\frac{8}{h(N)N^{1/2}} + \frac{4}{h(N) N^{3/2}} + \frac{2}{h(N)^2} \right)^{-1}.$$
   
When $f(N) \prec N^{3/2}$, we obtain
$$\mathbb P_f\big[ A_\Gamma \in A_{\mathcal B}\big]=\lim_{N \to \infty}\mathbb P\big[\Gamma \in \mathcal B \ | \ \Gamma \in \mathcal G^{N,f(N)}\big]
=\lim_{N \to \infty} \mathbb P \big[ X = 0\big]
=1-\lim_{N \to \infty} \mathbb P \big[ X \neq 0\big]
= 0.$$
   

Thus asymptotically almost surely $A_\Gamma$ is not $(2,2)$--free. In view of Lemma \ref{lem:comparing_dim2_2_2free}, this also means that asymptotically almost surely  $A_\Gamma$ is not of dimension $2$, this proves item \ref{item:smaller} in Theorem \ref{theorem:size_of_dimension2}. \\

We note that the above calculation allows us to find a lower bound for $\mathbb P_f \big[ A_\Gamma \in A_\mathcal B\big]$ at $f(N)=N^{3/2}$. Indeed, this is equivalent to saying that $h(N)=1$ and hence we get $\mathbb P \big[ X \neq 0\big] \gtrsim \frac{1}{3}$, and so at $f(N)=N^{3/2}$  we have $\mathbb P_f \big[ A_\Gamma \in A_\mathcal B\big] \leq \frac{2}{3}$. Hence by Lemma \ref{lem:comparing_dim2_2_2free}, this proves item \ref{item:equal} in the Theorem. \\

We note that  $\mathbb P\big[\Gamma \in \mathcal B \ | \ \Gamma \in \mathcal G^{N, f(N)}\big] =1-\mathbb P \big[ X \geq 1\big]$ and by the Markov inequality: $$ \mathbb P \big[ X \geq 1\big] \leq \mathbb E\big[X\big] \leq N^3f(N)^{-2}.$$ 

Hence if $f(N) \succ N^{3/2}$ then we can write $f(N)=N^{3/2}g(N)$ for some non-decreasing divergent function $g:\mathbf N \to \mathbf N$ and in this case $$\mathbb P \big[ X \geq 1\big] \leq \frac{1}{g(N)^2}. $$

Therefore, for $f(N) \succ N^{3/2}$ we have
$$\mathbb P_f\big[A_\Gamma \in A_\mathcal B\big] = \lim_{N \to \infty} \mathbb P\big[\Gamma \in \mathcal B \ | \ \mathcal G^{N, f(N)}\big]
=1-\lim_{N \to \infty} \mathbb P \big[ X \geq 1\big]
\geq 1-\lim_{N \to \infty} \frac{1}{g(N)^2}
=1.$$
In particular, asymptotically almost surely $A_{\Gamma}$ is $(2,2)$--free. By applying Lemma \ref{lem:comparing_dim2_2_2free} (as $f(N) \succ N$), we get that asymptotically almost surely $A_{\Gamma}$ is $2$-dimensional. This proves item \ref{item:bigger} and hence Theorem \ref{theorem:size_of_dimension2}.
 \end{proof}

We are now able to prove a refinement of Corollary \ref{Coro11Conj}:

\begin{corollary} \label{Coro7Conj}
Let $f : \mathbf{N} \rightarrow \mathbf{N}$ be a function satisfying $f(N) \succ N^{3/2}$. Then an Artin group $A_{\Gamma}$ picked at random (relatively to $f$) satisfies any of the following property asymptotically almost surely:
\begin{enumerate}
	\item \label{torsionfree} $A_{\Gamma}$ is torsion-free;
    \item\label{trivialcentre} $A_{\Gamma}$ has trivial centre;
    \item\label{conjpb} $A_{\Gamma}$ has solvable word and conjugacy problems;
    \item\label{Kpi1} $A_{\Gamma}$ satisfies the $K(\pi, 1)$-conjecture;
    \item\label{parabolic} The set of parabolic subgroups of $A_{\Gamma}$ is closed under arbitrary intersections;
    \item\label{AH} $A_{\Gamma}$ is acylindrically hyperbolic;
    \item\label{Titsalt} $A_{\Gamma}$ satisfies the Tits Alternative.
\end{enumerate}
\end{corollary}

\begin{proof}
By Theorem \ref{theorem:size_of_dimension2}, $A_{\Gamma}$ is asymptotically almost surely $2$-dimensional. Using Lemma \ref{lem:comparing_dim2_2_2free}, $A_{\Gamma}$ is also asymptotically almost surely $(2, 2)$-free. Using Lemma \ref{lem:being_reducible_uniformly small}, we also know that $A_{\Gamma}$ is asymptotically almost surely irreducible. Using Lemma \ref{lem:interescting_big_class} twice, this ensures that $A_{\Gamma}$ is asymptotically almost surely in the class
$$A_{\mathcal K} \coloneqq A_{Irr} \cap A_{\mathcal D} \cap A_{\mathcal B}.$$
 Note that the results given in Conjecture \ref{Conj1to11} for the points 1, 2, 3, 4, 5, 7 and 10 concern families of Artin groups that all contain $A_{\mathcal K}$. In particular, every Artin group of $A_{\mathcal K}$ satisfies the 7 points of this Corollary.   
\end{proof}
\noindent \textbf{What happens at } $f(N)=N^{3/2}$ \textbf{?}
\medskip 

Finding out the exact probability for an Artin group to be $2$-dimensional (or equivalently, $(2, 2)$--free) at $f(N)=N^{3/2}$ requires more work. In Theorem \ref{theorem:size_of_dimension2}, we gave an upper bound for this probability. The goal of the following lemma is to give an explicit formula for the value of $\mathbb P_f\big[A_\Gamma \in A_{\mathcal B} \big]$ at $f(N) = N^{3/2}$. Later, we give a conjecture on the exact value.
 
\begin{lemma}
\label{lem:proba_of_2,2_free}
	For all non-decreasing, divergent functions $f:\mathbf N \to \mathbf N$ we have that 
	$$ \mathbb P_f\big[ A_\Gamma \in A_{\mathcal B}\big] = \lim_{N \to \infty} \left(\frac{f(N)-1}{f(N)}\right)^{N \choose 2}\left(\sum_{k=1}^{\lfloor N/2 \rfloor}\frac{N!(f(N)-1)^{-k}}{(N-2k)!\,k! \,2^k}+1\right).$$
\end{lemma}

\begin{proof}


	
	Let $E_k$ be the family of defining graphs that have exactly $k$ edges labelled by a $2$, and consider the associated family $A_{E_k}$ of Artin groups. Note that each edge is attached to two vertices, so by the pigeonhole principle, if $k > N/2$ then $ \mathbb P_f \big[ \Gamma \in \mathcal B \cap E_k \big]=0$. Hence
	
	$$\mathbb P\big[\Gamma \in \mathcal B \ \vert \ \Gamma \in \mathcal G^{N, f(N)}  \big] = \sum_{k=0}^{\lfloor N/2 \rfloor}\mathbb P \big[\Gamma \in \mathcal B \cap E_k \ \vert \Gamma \in \mathcal G^{N, f(N)} \big].  $$

	As usual, the total number of graphs in $\mathcal G^{N, f(N)}$ is $f(N)^{N \choose 2}$. On the other hand, we must compute how many of these graphs have exactly $k$ edges labelled by a $2$, while these edges are never adjacent.

	First of all, when $k=0$, we have $\mathbb P \big[\Gamma \in \mathcal B \cap E_k \ | \ \Gamma \in \mathcal G^{N, f(N)} \big]=\left(\frac{f(N)-1}{f(N)}\right)^{N \choose 2}.$
	

    For the case when $0 < k \leq \lfloor N/2 \rfloor$, we look at how many ways we have of placing the $k$ edges labelled by a 2. For the first such edge, we have ${N \choose 2}$ choices. The two vertices of the first edge must not appear in any other edge labelled by a $2$, so for the second edge we only have ${N-2 \choose 2}$ choices left. This goes on until the $k$-th edge labelled by a 2, for which we have ${N-2(k-1) \choose 2}$ choices. As the order in which we have chosen these edges do not matter, we must divide this product by $k!$. Now for the remaining ${N \choose 2}-k$ edges, we can use any label other than a $2$. Hence we multiply the previous product by $(f(N)-1)^{{N \choose 2}-k}$. Hence, for $0 < k \leq \lfloor N/2 \rfloor$, we have
$$\mathbb P \big[\Gamma \in \mathcal B \cap E_k \ | \ \Gamma \in \mathcal G^{N, f(N)} \big]=\frac{(f(N)-1)^{{N \choose 2}-k} \cdot \prod_{i=0}^{k-1} {N-2i \choose 2}}{f(N)^{N \choose 2} \cdot k!}.$$
	Therefore:
	 \begin{align*}
 \begin{split}
 	\mathbb P_f \big[A_\Gamma \in A_{\mathcal B} \big] &= \lim_{N \to \infty} \sum_{k=1}^{\lfloor N/2 \rfloor}\mathbb P \big[\Gamma \in \mathcal B \cap E_k \ | \ \Gamma \in \mathcal G^{N, f(N)} \big]+\mathbb P \big[ \Gamma \in \mathcal B \cap E_0 \ | \ \Gamma \in \mathcal G^{N, f(N)} \big]\\
 	&= \lim_{N \to \infty} \sum_{k=1}^{\lfloor N/2 \rfloor}\frac{(f(N)-1)^{{N \choose 2}-k} \cdot \prod_{i=0}^{k-1} {N-2i \choose 2}}{f(N)^{N \choose 2} \cdot k!}+\left(\frac{f(N)-1}{f(N)}\right)^{N \choose 2} \\
 	&= \lim_{N \to \infty} \left(\frac{f(N)-1}{f(N)}\right)^{N \choose 2}\left(\sum_{k=1}^{\lfloor N/2 \rfloor}\frac{N!(f(N)-1)^{-k}}{(N-2k)!\,k! \,2^k}+1\right) \\
 \end{split}	
 \end{align*}
 
 where we go from the second to the third line by noting that $$ \prod_{i=0}^{k-1} {N-2i \choose 2} = \frac{1}{2^k}N(N-1)(N-2) \hdots (N-2(k-1))(N-2(k-1)-1)=\frac{N!}{(N-2k)!2^k}.$$ \end{proof}
 
Now, by Lemma \ref{lem:comparing_dim2_2_2free} at $f(N) =N^{3/2}$ we have $ \mathbb P_f \big[ A_\Gamma \in A_{\mathcal D} \big]=\mathbb P_f\big[A_\Gamma \in A_{\mathcal B} \big]$, hence Lemma \ref{lem:proba_of_2,2_free} also holds for $\mathbb P_f\big[A_\Gamma \in A_{\mathcal D} \big]$.
We have computed this expression in $\textit{Python}$ for $N$ up to 190, which leads us to the following conjecture.
\begin{conjecture}
For $f(N)=N^{3/2}$ we have:
	$$\mathbb P_f \big[ A_\Gamma \in A_\mathcal B]= 1-e^{-1}.$$
In particular, we also have:
$$\mathbb P_f \big[ A_\Gamma \in A_\mathcal D \big]= 1-e^{-1}.$$
\end{conjecture}

\section{Acylindrical hyperbolicity and centres.}

Two open questions in the study of Artin groups is whether all irreducible non-spherical Artin groups are acylindrically hyperbolic and have trivial centres (see Conjecture \ref{Conj1to11}). In this section, we study these two aspects of Artin groups for another family of Artin groups, that we will denote $A_{\mathcal{C}^C}$. The families of Artin groups studied in Section 2 and 3 are very large when $f(N)$ grows fast enough compared to $N$. While the spirit of this section resembles that of Sections 2 and 3, $A_{\mathcal{C}^C}$ will turn out to be very large when $f(N)$ grows slowly enough compared to $N$.


\begin{definition} 
A graph $\Gamma$ is said to be a cone if it has a join decomposition as a subgraph consisting of a single vertex $v_0$ and a subgraph $\Gamma'$ such that $ \Gamma =v_0 \ast \Gamma'$.
Let $\mathcal C$ be the class of defining graphs that are cones and $\mathcal C^C$ the class of defining graphs which are not cones.
\end{definition}

Recall that $Irr$ is the class of irreducible graphs. By \cite[Theorem 1.4]{KatoOguni}, we have that if $\Gamma$ has at least $3$ vertices, is irreducible and is not a cone then $A_\Gamma$ is acylindrically hyperbolic. Hence it suffices to find the probability that a random Artin group is irreducible and is not a cone.

\begin{proposition}
\label{prop:ah_cone}
For all $\alpha \in (0,1)$ and all non-decreasing functions $f(N) \prec N^{1-\alpha}$ we have that $$\mathbb P_f \big[ A_\Gamma \in A_{\mathcal C^C} \big]=1.$$
\end{proposition}

\begin{proof}
Fix $\alpha \in (0,1)$ and $f(N) \prec N^{1-\alpha}$ a non-decreasing divergent function. Then, by definition, there exists a non-decreasing divergent function $h$ such that $f(N)h(N)=N^{1-\alpha}$.


By the definition of a cone and by a union bound, we get:
\begin{align*}
\begin{split}
	\mathbb P\big[ \Gamma \in \mathcal C \mid \Gamma \in \mathcal G^{N, f(N)}\big] &\leq \sum_{v_0 \in V(\Gamma)} \mathbb P \big[\forall u \in  V(\Gamma)-{v_0}  : m_{u,v_0} \neq \infty \mid \Gamma \in \mathcal G^{N, f(N)}  \big] \\
	&= \sum_{v_0 \in V(\Gamma)} \left(\frac{f(N)-1}{f(N)}\right)^{N-1} \\
	&= N\left(\frac{f(N)-1}{f(N)}\right)^{N-1} \\
	& = N\left(\left(\frac{f(N)-1}{f(N)}\right)^{f(N)}\right)^{h(N)N^\alpha}\left(\frac{f(N)}{f(N)-1}\right). \\
\end{split}	
\end{align*}

Thus: 
$$ \mathbb P_f\big[A_\Gamma \in A_{\mathcal{C}}\big] =\lim_{N \to \infty} \mathbb P\big[ \Gamma \in \mathcal{C} \mid  \Gamma \in \mathcal G^{N, f(N)}\big] =\lim_{N \to \infty} Ne^{-N^{\alpha}h(N)}=0.$$

Hence for $f(N) \prec N^{1-\alpha}$ we have $\mathbb P_f \big[ A_\Gamma \in A_{\mathcal{C}^C}\big]=1$, proving the proposition.
\end{proof}

\begin{corollary}
\label{cor:cone_ah}
Let $\alpha \in (0,1)$ and let $f(N) \prec N^{1-\alpha}$ be a non-decreasing divergent function. Then an Artin group picked at random (relatively to $f$) asymptotically almost surely is acylindrically hyperbolic and has a trivial centre.
\end{corollary}

\begin{proof}
We note that by Lemma \ref{lem:being_reducible_uniformly small} and Lemma \ref{lem:interescting_big_class} we have  $\mathbb P_f\big[A_\Gamma \in A_{Irr} \cap A_{\mathcal C^C}\big]=  \mathbb P_f \big[ A_\Gamma \in A_{\mathcal{C}^C}\big]$. As we noted above, by \cite[Theorem 1.4]{KatoOguni}, if $\Gamma$ is irreducible and not a cone then $A_\Gamma$ is acylindrically hyperbolic. Hence, by Proposition \ref{prop:ah_cone}, for a function $f$ as in the statement of the Corollary, we get that a random Artin group (relatively to $f$) is asymptotically almost surely irreducible and a cone, hence asymptotically almost surely acylindrically hyperbolic.

Further, by \cite[Theorem 3.3]{CMW}, we have that if $\Gamma$ is irreducible and not a cone then $A_\Gamma$ has trivial centre. Hence a random Artin group (relatively to $f$) asymptotically almost surely has a trivial centre.
\end{proof}

Let $\alpha \in (0,1)$, by Corollary \ref{cor:cone_ah} and Corollary \ref{Coro7Conj}-\ref{AH},  we have shown that for all non-decreasing divergent functions $f$ such that either:
\begin{itemize}
    \item  $f(N) \prec N^{1-\alpha}$ 
    \item  $f(N) \succ N^{3/2}$
\end{itemize}
a random Artin group $A_\Gamma$ (relatively to $f$) is asymptotically almost surely acylindrically hyperbolic and has trivial centre. This motivates the following:

\medskip
\textbf{Question:} For which non-decreasing divergent functions $f$ do we have that a random Artin group (relatively to $f$) is asymptotically almost surely acylindrically hyperbolic and has trivial centre?\\

\bibliography{main}
\bibliographystyle{alpha}

\end{document}